\documentclass{amsart}
\usepackage{graphics}
\begin{document}

\newcommand{\Q}{\mathbb{Q}}
\newcommand{\C}{\mathbb{C}}
\newcommand{\D}{\mathbb{D}}
\newcommand{\Z}{\mathbb{Z}}
\newcommand{\R}{\mathbb{R}}
\newcommand{\N}{\mathbb{N}}
\newcommand{\XTC}{\hat{\mathbb{C}}}
\newcommand{\hatnu}{\nu_{\XTC}}
\newcommand{\hatsigma}{\sigma_{\XTC}}
\newcommand{\deltaXTC}{\delta_{\XTC}}
\newcommand{\deltaopen}{\delta_{\mathcal{O}(\XTC)}}
\newcommand{\deltaopenen}{\delta_{\mathcal{O}(\XTC)}^{en}}
\newcommand{\deltaclosed}{\delta_{\mathcal{C}(\XTC)}}
\newcommand{\deltamero}{\delta_{\mathfrak{M}(\XTC)}}
\newcommand{\cat}{^\frown}

\newcommand{\dom}{\operatorname{dom}}
\newcommand{\ran}{\operatorname{ran}}
\newcommand{\diam}{\operatorname{diam}}
\newcommand{\normal}{n}
\newcommand{\Log}{\operatorname{Log}}
\newcommand{\Arg}{\operatorname{Arg}}
\renewcommand{\Re}{\operatorname{Re}}
\renewcommand{\Im}{\operatorname{Im}}
\newcommand{\Int}{\operatorname{Int}}
\newcommand{\Ext}{\operatorname{Ext}}
\newcommand{\ray}[1]{\overrightarrow{#1}}

\newcommand{\const}{\mbox{\emph{Const.}} }

\newtheorem{theorem}{Theorem}[section]
\newtheorem{lemma}[theorem]{Lemma}
\newtheorem{assumption}[theorem]{Assumption}

\theoremstyle{definition}
\newtheorem{definition}[theorem]{Definition}
\newtheorem{question}[theorem]{Question}

\newtheorem{convention}[theorem]{Convention}

\newtheorem{notation}[theorem]{Notation}

\theoremstyle{theorem}
\newtheorem{corollary}[theorem]{Corollary}

\theoremstyle{theorem}
\newtheorem{proposition}[theorem]{Proposition}

\theoremstyle{theorem}
\newtheorem{givendata}[theorem]{Given Data}

\theoremstyle{theorem}
\newtheorem{conjecture}[theorem]{Conjecture}

\theoremstyle{definition}
\newtheorem{example}[theorem]{Example}

\theoremstyle{theorem}
\newtheorem{property}[theorem]{Property}

\numberwithin{equation}{section}

\title{The power of backtracking and the confinement of length}
\author{Timothy H. McNicholl}
\address{Department of Mathematics\\
Lamar University\\
Beaumont, Texas 77710}
\email{timothy.h.mcnicholl@gmail.com}


\begin{abstract}
We show that there is a point on a computable arc that does not belong to any computable rectifiable curve.  We also show that there is a point on a computable rectifiable curve with computable length that does not belong to any computable arc.
\end{abstract}
\maketitle

\section{Introduction}

The famous \emph{traveling salesman problem} is, given a finite list of cities and the distances between them, to determine if there is a route that visits each city exactly once and whose total distance is below some given bound.   A related problem is the \emph{Hamiltonian circuit problem}: given a finite graph, determine if it has a Hamiltonian circuit.  Both problems iare NP-complete \cite{Garey.Johnson.1979}.   

In 1990, P. Jones proposed and solved the \emph{analyst's traveling salesman problem} \cite{Jones.1990}.  Namely, which subsets of the plane are contained in a rectifiable curve?  In 1992, K. Okikiolu solved the $n$-dimensional version of this problem \cite{Okikiolu.1992}.  Such curves have finite length, but may backtrack.

Recently, X. Gu, J. Lutz, and E. Mayordomo considered and solved a computable version of this problem \cite{Gu.Lutz.Mayordomo.2006}.  Namely, which \emph{points} in $n$-dimensional Euclidean space belong to \emph{computable} rectifiable curves?  The results uncovered deep connections with topics in algorithmic randomness and effective Hausdorff dimension \cite{Downey.Hirschfeldt.2010}, \cite{Lutz.2005}.

 Here, we consider yet another effective version of the analyst's traveling salesman problem: which points belong to computable \emph{arcs}.  By an arc, we mean a homeomorphic image of $[0,1]$.  By a computable arc we mean one for which there is a computable homeomorphism.  Thus, an arc can not backtrack, but may have infinite length.  This problem is thus a computable and continuous variation on the Hamiltonian circuit problem.  It is well-known that there are computable arcs with infinite length.  In fact, many fractal curves have this property \cite{Falconer.2003}.
 
 Interestingly, the continuous version of the Hamiltonian circuit problem was solved in 1917 by R.L. Moore and J.R. Kline \cite{Moore.Kline.1919}.  Whereas the work of Jones and Okikiolu seems to translate more or less directly into the framework of computability, the methods of Kline and Moore so not appear to do so.
 
As suggested in \cite{Gu.Lutz.Mayordomo.2006}, these results can be interpreted in terms of the possible planar motion of a nanobot.  The effective version of the traveling salesman problem asks which points could be reached if the nanobot is allowed to backtrack but must follow a path of finite length.  In this case, since backtracking is allowed, the finite length requirement seems to be a reasonable way of excluding space-filling curves \cite{Hocking.Young.1988}.  The effective version of the Hamiltonian circuit problem asks which points could be reached if backtracking is prohibited but paths of infinite length are allowed.   
 
 Here, we compare these two classes of points.  Our conclusion is that neither is contained in the other.  Accordingly, our first main result is that there is a point on a computable arc that does not belong to any computable rectifiable curve.  
 Thus, if backtracking is prohibited but the traverse of an infinite length is allowed, there are points which could be reached by the nanobot that it could not reach even with backtracking once its path is required to have finite length.  
 We give two proofs which reveal connections with algorithmic randomness as well as effective Hausdorff dimension. 
 
  Our second main result is that there is a $\Delta^0_2$ point that belongs to a computable rectifiable curve \emph{with computable length} but that does not belong to any computable arc.  That is, if backtracking is allowed, then there are points that can be reached by a nanobot that could not be reached when backtracking is prohibited even if it must traverse a path whose length is a computable real.  The proof is a finite-injury priority argument \cite{Soare.1987} which also uses recent results on effective local connectivity and the computation of space-filling curves \cite{Daniel.McNicholl.2009}, \cite{Couch.Daniel.McNicholl.2010}.  We note that
Ko has shown there is a computable rectifiable curve whose length is not computable \cite{Ko.1995}.  Thus, the point we construct here belongs to a smaller class than that considered by Lutz \emph{et. al.}.
  
\section{Background and preliminaries from computability and computable analysis}

Our approach to computable analysis is more or less that which appears in \cite{Weihrauch.2000}.  We assume familiarity with the fundamental concepts of computability \cite{Boolos.Burgess.2002}, \cite{Cooper.2004}.

A \emph{rational interval} is an open interval whose endpoints are rational.
A \emph{rational rectangle} is a Cartesian product of two rational intervals.

A \emph{name} of a real number $x$ is a list of all rational intervals to which $x$ belongs.  A real is \emph{computable} if it has a computable name.

A \emph{name} of a point $p \in \R^2$ is a list of all rational rectangles to which $x$ belongs.

A  \emph{name} of an open set $U \subseteq \R^2$ is a list of all rational rectangles $R$ such that 
$\overline{R} \subseteq U$.

A set $X \subseteq \R^2$ is \emph{computably compact} if it is possible to uniformly compute from a number $k \in \N$ a finite covering of $X$ by rational rectangles each of which contains a point of $X$ and has diameter less than $2^{-k}$.  

A function $f : \subseteq \R \rightarrow \R^2$ is \emph{computable} if there is a Turing machine with a one-way output tape and with the property that whenever a name of an $x \in \R$ is written on its input tape and the machine is allowed to run indefinitely, a name of $f(x)$ is written on the output tape.  Equivalently, there is an \emph{$e$-reduction} (\cite{Cooper.2004}) that maps each name of any $x \in \dom(f)$ to a name of $f(x)$.

A subset of the plane is a \emph{curve} if it is the range of a continuous function whose domain is $[0,1]$.   
If at least one such function is computable, then the curve is also said to be computable.  
A subset of the plane is an \emph{arc} if it is a homeomorphic image of $[0,1]$.  
An arc is computable if at least one such homeomorphism is computable.  

Every arc is a curve.  Thus, there is a clash of terminology here since there are arcs which are computable \emph{as curves} but for which there is no computable homeomorphism with $[0,1]$ \cite{Gu.Lutz.Mayordomo.2009}.  However, context will always make clear which sense of `computable' is intended.  That is, whenever we are considering a curve, if the curve is also an arc, it will be called computable just in case it is the image of $[0,1]$ under a computable homeomorphism.  Otherwise, it will be called computable just in case it is the image of $[0,1]$ under a computable map.

A curve is \emph{computably rectifiable} if it has finite length and its length is a computable real.

Since $[0,1]$ is compact, a function $f : [0,1] \rightarrow \R^2$ is a homeomorphism if and only if it is continuous and injective.  See, \emph{e.g.} Theorem 2-103 of \cite{Hocking.Young.1988}.

A set $X \subseteq \R^2$ is \emph{effectively locally connected} if there is a Turing machine with a one-way output tape and two input tapes and with the property that whenever a name of a point $p \in X$ is written on the  first input tape and a name of an open set $U \subseteq \R^2$ that contains $p$ is written on the second input tape and the machine is allowed to run indefinitely, a name of an open set $V \subseteq \R^2$ for which $p \in V$ and $V \cap X$ is connected is written on the output tape.  See \cite{Miller.2004}, \cite{Brattka.2008}, \cite{Daniel.McNicholl.2009}, \cite{Couch.Daniel.McNicholl.2010}.  The following is a consequence of the main theorem of \cite{Couch.Daniel.McNicholl.2010} and is an effective version of the Hahn-Mazurkiewicz Theorem \cite{Hocking.Young.1988}.

\begin{theorem}\label{thm:PEANO}
If $X \subseteq \R^2$ is computably compact and effectively locally connected, then there is a computable surjection of $[0,1]$ onto $X$.
\end{theorem}

We will use a number of concepts from \cite{Daniel.McNicholl.2009}.  However, as we will only work in $\R^2$, we specialize their definitions as follows.  To begin, a sequence of sets $(U_1, \ldots, U_n)$ is a \emph{chain} if $U_i \cap U_{i+1} \neq \emptyset$ whenever $1 \leq i <n$ and is a \emph{simple chain} if $U_i \cap U_j \neq \emptyset$ precisely when $|i - j| = 1$.

A simple chain $(V_1, \ldots, V_l)$ in a topological space \emph{refines} a simple chain $(U_1, \ldots, U_k)$ if there are numbers $s_1, \ldots, s_k$ such that $\sum s_j = l$, $s_j \geq 2$, 
\begin{eqnarray*}
\overline{V_1}, \ldots, \overline{V_{s_1}} & \subseteq & U_1\mbox{, and}\\
\overline{V_{s_1 + \ldots + s_{i-1} + 1}}, \ldots, \overline{V_{s_1 + \ldots + s_i}} & \subseteq & U_i\ \mbox{whenever $1 < i \leq k$.}\\
\end{eqnarray*}
We say that this refinement \emph{has type $(s_1, \ldots, s_k)$}. 
 
A \emph{witnessing chain} is a chain of rational rectangles.  We denote witnessing chains by $\omega$ and its accented and subscripted variants.  If $\omega = (R_1, \ldots, R_k)$ is a witnessing chain, then we let:
\begin{eqnarray*}
k_\omega & = & k\\
R_{\omega, i} & = & R_i\\
V_\omega & = & \bigcup_i R_i
\end{eqnarray*}

An \emph{arc chain} is a sequence of witnessing chains $(\omega_1, \ldots, \omega_l)$ such that 
$(V_{\omega_1}, \ldots, V_{\omega_l})$ is a simple chain.  We denote arc chains by $\mathfrak{p}$ and its accented and subscripted variants.  If $\mathfrak{p} = (\omega_1, \ldots, \omega_l)$ is an arc chain, then we let:
\begin{eqnarray*}
l_\mathfrak{p} & = & l\\
\omega_{\mathfrak{p}, i} & = & \omega_i\\
V_{\mathfrak{p}, i} & = & V_{\omega_i}\\
V_\mathfrak{p} & = & \bigcup_i V_{\omega_i}
\end{eqnarray*}
We also define the \emph{diameter of $\mathfrak{p}$} to be 
\[
\max_j \diam(V_{\omega_j}).
\]
We denote this quantity by $\diam(\mathfrak{p})$.

We say that an arc chain $\mathfrak{p}_0$ \emph{refines} an arc chain $\mathfrak{p}_1$ if 
$(V_{\mathfrak{p}_0, 1}, \ldots, V_{\mathfrak{p}_0, l_{\mathfrak{p}_0}})$
refines $(V_{\mathfrak{p}_1, 1}, \ldots, V_{\mathfrak{p}_1, l_{\mathfrak{p}_1}})$.
We say this refinement has type $(s_1, \ldots, s_{l_{\mathfrak{p}_1}})$ if the refinement of 
$(V_{\mathfrak{p}_1, 1}, \ldots, V_{\mathfrak{p}_1, l_{\mathfrak{p}_1}})$ by 
$(V_{\mathfrak{p}_0, 1}, \ldots, V_{\mathfrak{p}_0, l_{\mathfrak{p}_0}})$ has this type. 

A sequence of arc chains (possibly finite) $\mathfrak{p}_0, \mathfrak{p}_1, \ldots$ is \emph{descending} if 
$\mathfrak{p}_{i+1}$ refines $\mathfrak{p}_i$.

We will need the following.

\begin{theorem}\label{thm:COMPUTABLE.ARC}
If an arc $A \subseteq \R^2$ is computable, then there is an infinite, computable, and descending sequence of arc chains $\mathfrak{p}_0, \mathfrak{p}_1, \ldots$ such that $A = \bigcap_i V_{\mathfrak{p}_i}$ and $\diam(\mathfrak{p}_i) < 2^{-i}$.  Furthermore, if 
$\mathfrak{p}_0, \mathfrak{p}_1, \ldots$ is such a sequence of arc chains, then 
$\bigcap_i V_{\mathfrak{p}_i}$ is a computable arc.
\end{theorem}

Theorem \ref{thm:COMPUTABLE.ARC} is an immediate consequence of the results in \cite{Daniel.McNicholl.2009}.  As these results have not yet appeared in print, we include a proof here for which we will need some additional notions.

\begin{definition}\label{def:LABELLED}
A \emph{labelled arc chain} consists of a finite sequence of pairs \\
$((\omega_1, I_1), \ldots, (\omega_l, I_l))$ such that 
$(\omega_1, \ldots, \omega_l)$ is an arc chain, $(I_1, \ldots, I_l)$ is a simple chain of closed
subintervals of $[0,1]$ whose endpoints are rational, and $[0,1] = \bigcup_j I_j$.
\end{definition}

If $\Lambda = ((\omega_1, I_1), \ldots, (\omega_l, I_l))$ is a labelled arc chain, we let:
\begin{eqnarray*}
\mathfrak{p}_\Lambda & = & (\omega_1, \ldots, \omega_l)\\
I_{\Lambda,j } & = & I_j\\
\omega_{\Lambda, j} & = & \omega_j\\
l_\Lambda & = & l
\end{eqnarray*}

\begin{definition}\label{def:LABELLED.REFINES}
Let $\Lambda_0$, $\Lambda_1$ be labelled arc chains.  We say that $\Lambda_0$ \emph{refines} $\Lambda_1$ if there are numbers $s_1, \ldots, s_{l_{\Lambda_1}}$ such that $\mathfrak{p}_{\Lambda_0}$ is a refinement of 
$\mathfrak{p}_{\Lambda_1}$ of type $(s_1, \ldots, s_{l_{\Lambda_1}})$ and 
$(I_{\Lambda_0, 1}, \ldots, I_{\Lambda_0, l_{\Lambda_0}})$ is a refinement of 
$(I_{\Lambda_1, 1}, \ldots, I_{\Lambda_1, l_{\Lambda_1}})$ of the same type.  Again, we refer to 
$(s_1, \ldots,  s_{l_{\Lambda_1}})$ as the \emph{type} of this refinement.
\end{definition}

\begin{definition}\label{def:DIAM.LABELLED}
If $\Lambda$ is a labelled arc chain, then the \emph{diameter} of $\Lambda$ is 
\[
\max\{\diam(\mathfrak{p}_\Lambda), \diam(I_{\Lambda, 1}), \ldots, \diam(I_{\Lambda, l_\Lambda})\}.
\]
We denote this by $\diam(\Lambda)$.
\end{definition}

We say that a sequence of labelled arc chains $\Lambda_0, \Lambda_1, \ldots$ (finite or infinite) is \emph{descending} if $\Lambda_{i+1}$ refines $\Lambda_i$.

The proof of Theorem \ref{thm:COMPUTABLE.ARC} now turns on the following three lemmas.

\begin{lemma}\label{lm:ARC.CHAIN}
Suppose $\{\Lambda_j\}_{j \in \N}$ is an infinite descending sequence of labelled arc chains such that 
$\lim_{j \rightarrow \infty} \diam(\Lambda_j) = 0$.  
\begin{enumerate}
	\item For each $x \in [0,1]$, the set 
	\[
	N_x =_{df} \bigcap_{j=0}^\infty \bigcup \{V_{\omega_{\Lambda_j, i}}\ :\ x \in I_{\Lambda_{j,i}}\}
	\]
	contains exactly one point.
	
	\item If $\{\Lambda_j\}_{j \in \N}$ is computable, and if for each $x \in [0,1]$ we define $f(x)$ to be the unique point of $N_x$, then $f$ is a computable injective map.
\end{enumerate}
\end{lemma}

\begin{proof}
For each $j$, there are at most two values of $i$ for which $x \in I_{\Lambda_j, i}$.  Furthermore, if there are two such values of $i$, then they differ by $1$.  Abbreviate $V_{\omega_{\Lambda_j, i}}$ by $V_{j,i}$, $I_{\Lambda_j, i}$ by $I_{j,i}$, and $l_{\Lambda_j}$ by $l_j$.  

Let $j_0 < j_1$.  
We show that 
\[
\overline{\bigcup\{V_{j_1, i}\ :\ x \in I_{j_1, i}\}} \subseteq \bigcup\{V_{j_0, i}\ :\ x \in I_{j_0, i}\}.
\]

Let $i_k$ be the smallest number such that $x \in I_{j_k, i_k}$.  
$\Lambda_{j_1}$ refines $\Lambda_{j_0}$;  let $(s_1, \ldots, s_{l_{j_0}})$ be the type of this refinement.  
Let $S_0 = \{1, \ldots, s_1\}$, and let $S_i = \{s_1 + \ldots s_{i-1}+1, \ldots, s_1 + \ldots s_i \}$ whenever $1 < i \leq l_{j_0}$.
There is a unique $u$ such that $i_1 \in S_u$.  Hence, 
$I_{j_1, i_1} \subseteq I_{j_0, u}$ and $\overline{V_{j_1, i_1}} \subseteq V_{j_0, u}$.  Furthermore, $x \in I_{j_0, u} \cap I_{j_0, i_0}$.
If $x \not \in V_{j_1, i_1 + 1}$, then we are done.  Suppose $x \in V_{j_1, i_1 + 1}$.  
If $i_1 + 1 \in S_u$, then we are done.  Otherwise, it must be that $i_1 + 1 \in S_{u+1}$ and we are done.

It now follows that $N_x$ is non-empty.  Since $\lim_{j \rightarrow \infty} \diam(\mathfrak{p}_{\Lambda_j}) = 0$, 
it follows that $N_x$ contains exactly one point.

Now, suppose $\{\Lambda_j\}_j$ is computable.  Given a name of an $x \in [0,1]$, we list a rational rectangle $R$ if and only if there exist $I, j,i$ such that $x \in I$ and 
\[
R \supseteq \bigcup\{V_{j,i}\ :\ I \cap I_{j,i} \neq \emptyset\}.
\]
It follows that all and only those rational rectangles containing $f(x)$ are listed.  
Hence, $f$ is computable.

Suppose $0 \leq x_1 < x_2 \leq 1$.  Thus, there is a $j_0$ such that 
$|i_1 - i_2| \geq 2$ whenever $j \geq j_0$, $x_1 \in I_{j, i_1}$, and $x_2 \in I_{j, i_2}$.
Let 
\[
S_{x,j} = \bigcup\{V_{j,i}\ :\ x \in I_{j,i}\}.
\]
By definition, $f(x) \in S_{x,j}$.  However, whenever $j \geq j_0$, $S_{x_1, j} \cap S_{x_2, j} = \emptyset$.  So, $f(x_1) \neq f(x_2)$.
\end{proof}

\begin{lemma}\label{lm:LABEL}
If $\{\mathfrak{p}_j\}_{j \in \N}$ is a computable and descending sequence of arc chains for which 
$\lim_{j \rightarrow \infty} \diam(\mathfrak{p}_j) = 0$, then there is a computable and descending sequence of 
arc chains $\{\Lambda_j\}_{j \in \N}$ for which $\mathfrak{p}_{\Lambda_j} = \mathfrak{p}_j$ and such that 
$\lim_{j \rightarrow \infty} \diam(\Lambda_j) = 0$.
\end{lemma}

\begin{proof}
Let $I_{\Lambda_0, 1}, \ldots, I_{\Lambda_0, l_{\Lambda_0}}$ be the uniform partition
of $[0,1]$ into $l_{\Lambda_0}$ subintervals.

Suppose $I_{\Lambda_j, 1}, \ldots, I_{\Lambda_j, l_{\Lambda_j}}$ have been defined.
Suppose $\mathfrak{p}_{j+1}$ is a refinement of $\mathfrak{p}_j$ of type 
$(s_1, \ldots, s_k)$.  Let $I_{\Lambda_{j+1}, 1}, \ldots, I_{\Lambda_{j+1}, s_1}$ be the uniform 
partition of $I_{\Lambda_j, 1}$ into $s_1$ subintervals.  Let 
$I_{\Lambda_{j+1}, (s_1 + \ldots + s_{i-1} + 1)}, \ldots, I_{\Lambda_{j+1}, (s_1 + \ldots + s_i)}$
be the uniform partition of $I_{\Lambda_j, i}$ into $s_i$ subintervals when $1 < i \leq k$.
Since each $s_i$ is at least $2$, the conclusion follows.
\end{proof}

\begin{lemma}\label{lm:CHAIN}
Suppose $A$ is an arc and $\epsilon > 0$.  Then, there is a witnessing chain 
$\omega = (R_1, \ldots, R_k)$ such that $A \subseteq V_\omega$, $R_j \cap A \neq \emptyset$, and $\diam(R_j) < \epsilon$.
\end{lemma}

\begin{proof}
Since $A$ is compact, there are rational rectangles $S_1, \ldots, S_l$ that cover $A$, whose diameters are smaller than $\epsilon$, and each of which contains a point of $A$.  By Theorem 3-4 of \cite{Hocking.Young.1988}, there exist $R_1, \ldots, R_k \in \{S_1, \ldots, S_l\}$ such that 
$(R_1 \cap A, \ldots, R_k \cap A)$ is a simple chain that covers $A$.  Hence, 
$(R_1, \ldots, R_k)$ is a chain that covers $A$.
\end{proof}

\noindent\it Proof of Theorem \ref{thm:COMPUTABLE.ARC}:\rm\ Suppose $A \subseteq \R^2$ is a computable arc.  Let $f$ be a computable injective map of $[0,1]$ onto $A$.

For each $j \in \N$, let $I_{j,1}, \ldots, I_{j,2^j}$ be the uniform partition of $[0,1]$ into subintervals of 
length $2^{-j}$.  Since $f$ is injective, $f[I_{j, i_1}] \cap f[I_{j, i_2}] = \emptyset$ whenever $i_1 - i_2| \geq 2$.  It follows from Lemma \ref{lm:CHAIN} that for each $j,i$ we can choose a witnessing chain $\omega_{j,i}$ such that $f[I_{j,i}]$ is contained in $V_{\omega_{j,i}}$ and 
$V_{\omega_{j, i_1}} \cap V_{\omega_{j, i_2}} = \emptyset$ whenever $|i_1 - i_2| \geq 2$.
Furthermore, we can choose each $\omega_{j,i}$ so that $\diam(\omega_{j,i}) < 2^{-j}$.  Furthermore, we can ensure that 
\[
I_{j+1, i_1} \subseteq I_{j, i_2}\ \Rightarrow\ \overline{V_{\omega_{j+1, i_1}}} \subseteq V_{\omega_{j, i_2}}.
\]
We can then set 
\[
\mathfrak{p}_j = (\omega_{j,1}, \ldots, \omega_{j, 2^j}).
\]

The other direction follows from Lemmas \ref{lm:ARC.CHAIN} and \ref{lm:LABEL}.

\section{The confinement of length}

In this section, we give two proofs of the following.

\begin{theorem}\label{thm:NOT.ON.RECT}
There is a point on a computable arc that does not belong to any computable rectifiable curve.
\end{theorem}

For the first proof, we employ some concepts from algorithmic randomness.
Call an open set $U \subseteq \R^2$ \emph{computably enumerable} (\emph{c.e.}) if there is a computable sequence of rational rectangles $\{R_n\}_{n \in \N}$ such that $U = \bigcup_n R_n$.  An infinite sequence $\{U_n\}_{n \in \N}$ of open sets is \emph{uniformly c.e.} if there is a computable double sequence of rational rectangles $\{R_{n,k}\}_{n,k \in \N}$ such that 
$U_n = \bigcup_k R_{n,k}$ for each $n$.   A set $X \subseteq \R^2$ is said to have \emph{effective measure zero} if there is a descending uniformly \emph{c.e.} sequence of open sets $\{U_n\}_{n \in \N}$ such that 
the Lebesgue measure of each $U_n$ is at most $2^{-n}$.  Such a sequence of open sets is said to be a \emph{Martin-L\"of test}.  A point $p \in \R^2$ is \emph{Martin-L\"of random} if
$\{p\}$ does not have effective measure zero.

It follows from the constructions of Osgood (see \cite{Sagan.1994}) that for each $k \geq 1$, there is a computable arc $A \subseteq [0,1] \times [0,1]$ with measure $1 - 2^{-k}$.  The complements of these arcs then yield a Martin-L\"of test (after taking intersections with $(0,1) \times (0,1)$).  Thus, every Martin-L\"of random point in $(0,1) \times (0,1)$ lies on a computable arc.  However, by the results in \cite{Gu.Lutz.Mayordomo.2006}, no Martin-L\"of random point lies on a computable rectifiable curve.  

A pleasant side-effect of this proof is that \it every Martin-L\"of random point lies on a computable arc\rm. 

For our second proof of Theorem \ref{thm:NOT.ON.RECT}, we use effective Hausdorff dimension.  Let $A$ denote the quadratic von Koch curve of type 2 (see, \emph{e.g.}, \cite{Falconer.2003}).  Then, $A$ is a computable arc.  At the same time, the Hausdorff dimension of $A$ is 3/2.  Hence, there is a point on $A$ whose effective Hausdorff dimension is larger than $1$ (see \cite{Lutz.2003}).  However, as is pointed out in \cite{Gu.Lutz.Mayordomo.2006}, a point on a computable rectifiable curve has effective Hausdorff dimension at most $1$. \footnote{I am very grateful to Jack Lutz for showing me this line of argument.}

An interesting consequence of this proof is that even very natural fractal curves which are easy to visualize contain points which do not lie on any computable rectifiable curve.

\section{The power of backtracking}

Let $\{\mathfrak{p}_i^e\}_{e \in \N, i \leq k_e}$ be an effective enumeration of all computable descending sequences of arc chains.  That is:
\begin{itemize}
	\item $i,e \mapsto \mathfrak{p}_i^e$ is a computable partial function.
	
	\item $k_e \leq \aleph_0$.
	
	\item $\mathfrak{p}_i^e$ is defined if and only if $i < k_e$.
	
	\item $\mathfrak{p}_{i+1}^e$ refines $\mathfrak{p}_i^e$ if $i + 1< k_e$.
	
	\item $\diam(\mathfrak{p}_i^e) < 2^{-i}$ if $i < k_e$.
\end{itemize}

Let 
\[
A_e = \bigcap_{i < k_e} V_{\mathfrak{p}_i^e}.
\]

By Theorem \ref{thm:COMPUTABLE.ARC}, an arc $A \subseteq \R^2$ is computable if and only if there is an $e \in \N$ such that $k_e = \aleph_0$ and $A = A_e$.

Although the proof of the following is a simple exercise, it will be useful enough in the proof of Theorem \ref{thm:POWER.RETRACE} to state it as a proposition.

\begin{proposition}\label{prop:LIMIT.POINTS}
If $S$ is a connected subset of a topological space, and if $X \subseteq \overline{S} - S$, then $S \cup X$ is connected.
\end{proposition}

\begin{theorem}\label{thm:POWER.RETRACE}
There is a computably rectifiable and computable curve $C$ that contains a point $p$ which does not belong to any computable arc.
\end{theorem}

\begin{proof}
We construct $C$ and $p$ by stages.  Let $X[t]$ denote the value of $X$ at stage $t$.
For each $e \in \N$, let $R_e$ be the requirement
\[
k_e = \aleph_0\ \Rightarrow\ p \not \in A_e.
\]

\noindent\bf Stage $0$:\rm\ Let $C_0$ and $S_{0, e}$ be as in Figure \ref{fig:STAGE.0}.
\begin{figure}
\resizebox{3.5in}{3.5in}{\includegraphics{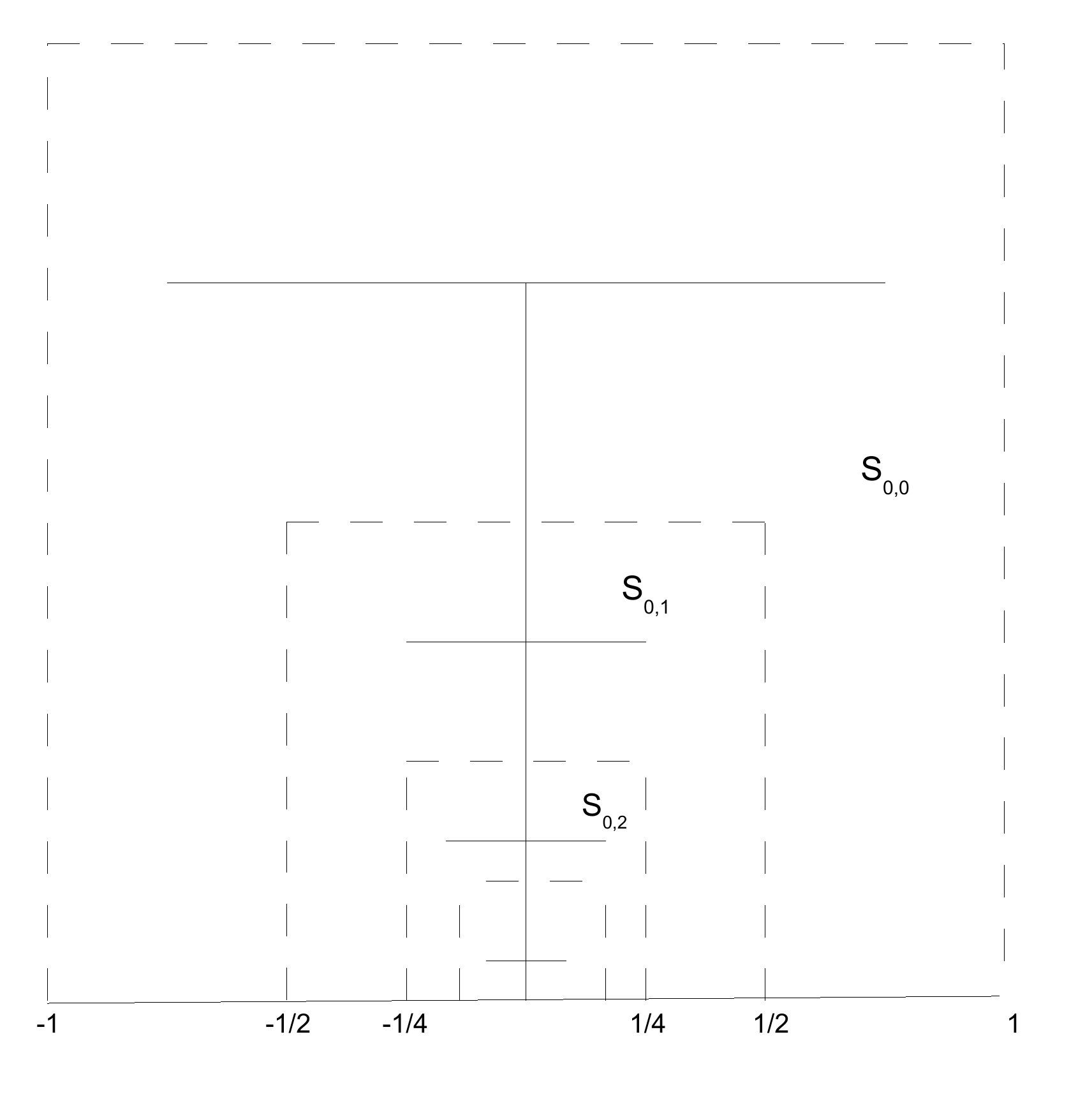}}
\caption{}\label{fig:STAGE.0}
\end{figure}
More formally, let:
\begin{eqnarray*}
S_{0,e} & = & [(-2^{-e}, -2^{-(e+1)}) \cup (2^{-(e+1)}, 2^{-e})] \times (2^{-e}, 2^{-(e-1)}).
\end{eqnarray*}
\begin{eqnarray*}
C_0 & = & (\{0\} \times [0, \frac{3}{2}]) \cup \bigcup_{j=1}^\infty ([-2^{-j}, 2^{-j}] \times \{ \frac{3}{2^j})\\
p_0 & = & (0,0)\\
\end{eqnarray*}
Let $h_{e,0}, h_{e,1}, \ldots$ denote in order of decreasing length the horizontal line segments in Figure \ref{fig:STAGE.0} that lie above the $x$-axis.  Let $q_{e,0}$ denote the midpoint of $h_{e,0}$.   
Let $v_0$ denote the line segment from $p_0$ to $q_{0,0}$.  \\

\noindent\bf Stage $t+1$:\rm\ If there do not exist $e,i$ and a rational rectangle $R$ such that 
\begin{itemize}
	\item there is no $j$ such that $\mathfrak{p}^e_j[t]\downarrow$ and $S_{t, e+1} \cap V_{\mathfrak{p}^e_j} = \emptyset$, 
 	\item $\mathfrak{p}_i^e[t]\downarrow$, 
 	\item $\overline{R} \subseteq S_{t,e}$, 
	\item $R \cap C_t$ is a line segment, 
	\item $R \cap \overline{V_{\mathfrak{p}_i^e}} = \emptyset$, and 
	\item $\diam(R) < 2^{-t}$, 
\end{itemize}
then let $S_{t+1,e} = S_{t,e}$, $p_{t+1} = p_t$, and $C_{t+1} = C_t$.   Otherwise, choose the least such $e,i, R$.  Let $l_t = R \cap C_t$.  For $e' \leq e$, let $S_{t+1, e'} = S_{t, e'}$.  Assume without loss of generality that $l_t$ is horizontal.  Define $D_{t+1}$ and $S_{t+1, e'}$ for $e' > e$ as in Figure \ref{fig:STAGE.t+1}.
\begin{figure}
\resizebox{3.5in}{3.5in}{\includegraphics{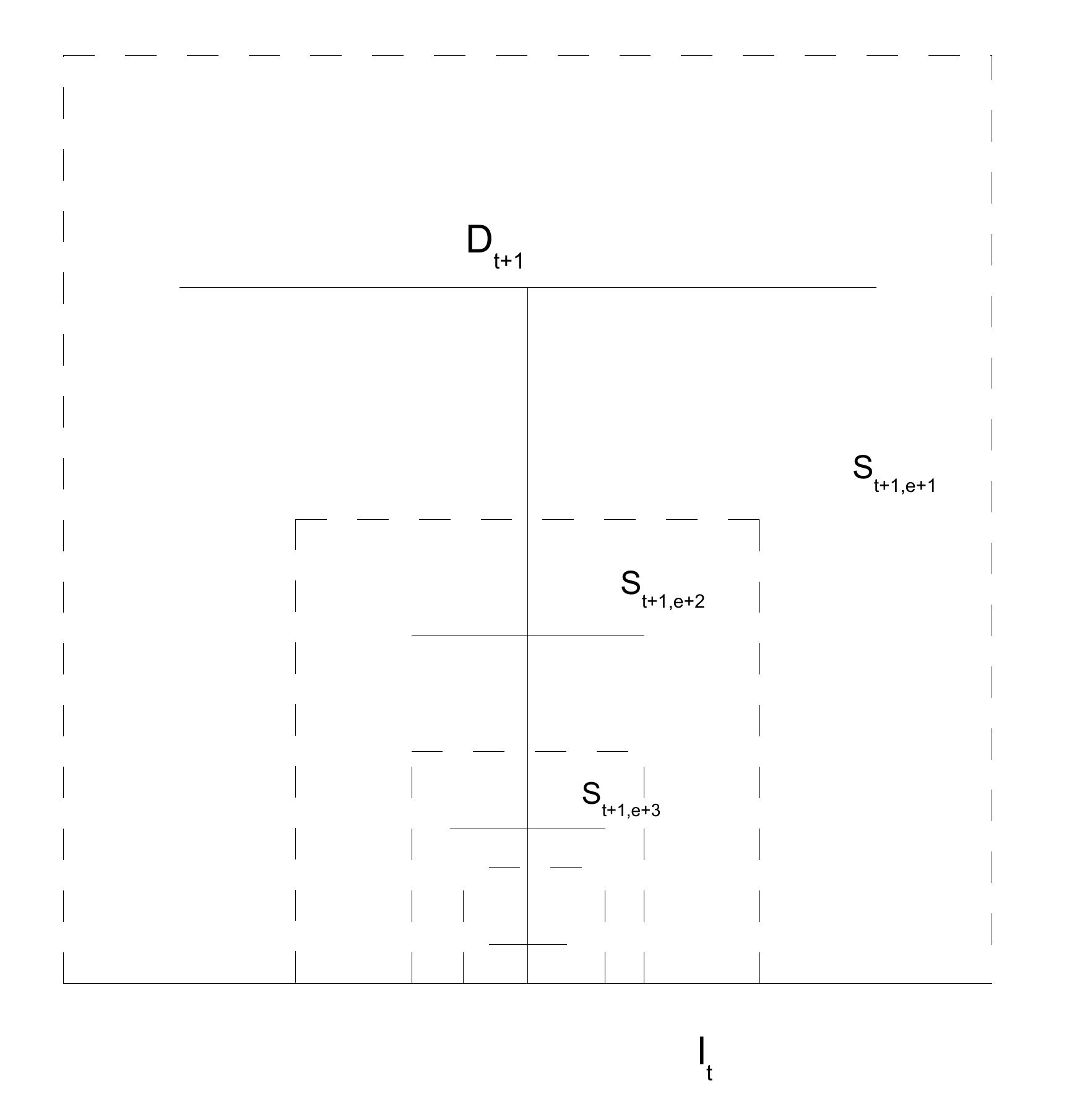}}
\caption{}\label{fig:STAGE.t+1}
\end{figure}
$S_{t+1, e+1}$ is chosen so that its closure is contained in $R$.  
Let $h_{e+1, t+1}, h_{e+2, t_1}, \ldots$ denote in order of decreasing length the line horizontal line segments in Figure \ref{fig:STAGE.t+1} that lie above $l_t$.  $D_{t+1}$ is constructed so that the length of $h_{e', t+1}$ is smaller than 
$2^{-t}2^{-(e' - e) -1}$ whenever $e' > e$.  Let $v_{t+1}$ denote the vertical line segment in Figure \ref{fig:STAGE.t+1}, and let $p_{t+1}$ denote the point it has in common with $l_t$.  
We let $C_{t+1} = C_t \cup D_{t+1}$.
We say that $R_e$ \emph{acts at stage $t+1$}.

This completes the description of the construction.  We divide its verification into a sequence of lemmas.  For all $e$, let $S_e =_{df} \lim_{t \rightarrow \infty} S_{t,e}$.

\begin{lemma}\label{lm:SAT}
Let $e \in \N$.  Then:
\begin{enumerate}
	\item $R_e$ acts at most finitely often.
	
	\item $S_e$ exists.
	
	\item If $k_e = \aleph_0$, then $S_{e+1}$ exists and contains no point of 
	$A_e$.
\end{enumerate}
\end{lemma}

\begin{proof}
By way of induction, suppose each $R_e'$ with $e' < e$ acts at most finitely often. 
Let $s_0$ be the least stage such that no $R_{e'}$ with $e' < e$ acts at any stage $s \geq s_0$.

Let us say that $S_{e'}$ has \emph{settled at stage $t$} if $S_{t,e'} = S_{t',e'}$ for all $t' \geq t$.
It follows that $s_0$ is the least stage at which $S_e$ has settled.  If $s_0 \geq 1$, then some $R_{e'}$ with $e' < e$ acts at $s_0 - 1$.   It then follows from the construction that $S_e \cap C_{s_0}$ is either a cross or a `T' as in Figure \ref{fig:POSSIBILITIES}.
\begin{figure}
\resizebox{3.5in}{3.5in}{\includegraphics{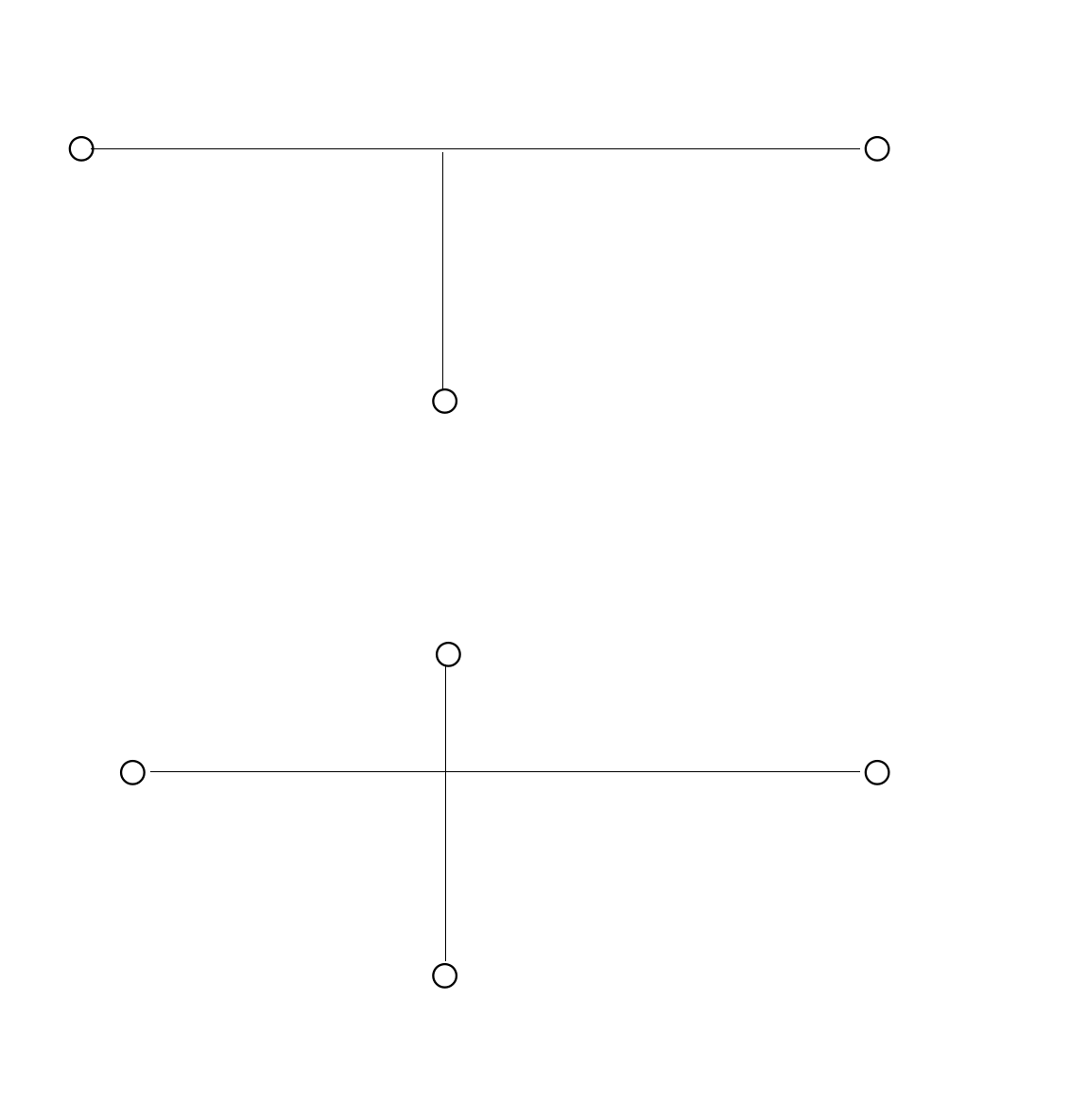}}
\caption{}\label{fig:POSSIBILITIES}
\end{figure}
We now show that there is at most one stage $t+1 \geq s_0$ at which $R_e$ acts.  For, suppose $t+1$ is the least number such that $t+1 \geq s_0$ and $R_e$ acts at $t+1$.  By construction, in particular by the action performed at $t+1$,  there is a number $i$ such that $\mathfrak{p}_i^e[t]\downarrow$ and $\overline{S_{t+1, e+1}} \cap V_{\mathfrak{p}_i^e} = \emptyset$.  It now follows by induction and the conditions which must be met in order for $R_e$ to act that $S_{t', e+1} = S_{t+1, e+1}$ for all $t' \geq t+1$ and that $R_e$ does not act at any $t' \geq t+1$.  It now also follows that $S_{e+1}$ exists. 

Suppose $k_e = \aleph_0$.  We show that $S_{e+1} \cap A_e = \emptyset$.  We first consider the case where there exists $t+1 \geq s_0$ such that $S_{t, e+1} \cap V_{\mathfrak{p}_j^e} = \emptyset$ for some $j$ such that $\mathfrak{p}_j^e[t]\downarrow$.  Hence, $R_e$ does not act at $t+1$.  It then follows by induction that $R_e$ does not act at any $t' \geq t+1$.  Hence, $S_{e+1} = S_{t+1, e+1}$.  Thus, $S_{e+1} \cap A_e = \emptyset$.

So, suppose that for every $t+1 \geq s_0$, $S_{t, e+1} \cap V_{\mathfrak{p}_j^e} \neq \emptyset$ for every $j$ such that $\mathfrak{p}_j^e[t]\downarrow$.  As noted, $C_{s_0}$ is either a cross or a `T'.  In either case, $C_{s_0}$ contains a unique point $q$ with the property that $C_{s_0} - \{q\}$ has at least three connected components.  No connected subset of an arc can have this property.  Hence, $C_{s_0} \not \subseteq A_e$.  Since, $\R^2 - A_e$ is open, there is a point $q_1 \in C_{s_0} - \{q\}$ that does not belong to $A_e$.  Since $q_1 \in S_e$ which is open, it follows that $R_e$ acts at some $t+1 \geq s_0$.  It then follows by what has just been shown that $S_{e+1}$ settles at $t+1$ and contains no point of $A_e$.
\end{proof}

Let 
\[
C = \overline{ \bigcup_s C_s}.
\]

Let $d$ denote the Euclidean metric on $\R^2$.  Let $D_r(q)$ denote the open disk of radius $r$ and center $q$.  When $X, Y \subseteq \R^2$, let 
$d(X,Y) = \inf_{q_1 \in X, q_2 \in Y} d(q_1, q_2)$.  When $q \in \R^2$ and $X \subseteq \R^2$, let 
$d(q, X) = d(\{q\}, X)$.  

\begin{lemma}\label{lm:SAT2}
$p =_{df} \lim_{t \rightarrow \infty} p_t$ exists, is $\Delta^0_2$, and does not belong to any computable arc though it does belong to $C$.
\end{lemma}

\begin{proof}
Suppose $R_{e_0}$ acts at $t_0 + 1$ and also that no $R_{e'}$ with $e' \leq e_0$ acts at any $t+1$ with $t > t_0$.  It follows from the construction that $d(p_{t+1}, p_{t_0 + 1}) < 2^{-t_0}$.  
There are infinitely many such $t_0$.  Hence, $\{p_t\}_t$ is a Cauchy sequence.
It also follows from the Limit Lemma that $p$ is $\Delta^0_2$.  Since $p_t \in C_t$ for all $t$, it follows that 
$p \in C$.  

It also follows that $S_{e_0 + 1}$ settles at $t_0 + 1$.  It then follows from Lemma \ref{lm:SAT} that $p$ does not belong to any computable arc.
\end{proof}

\begin{lemma}\label{lm:COMP.COMPACT}
$C$ is computably compact.
\end{lemma}

\begin{proof}
By construction, if $q \in C - C_0$, then $q \in S_{e,0}$ for some $e$.  Hence, 
$d(q, C_0) < 2$.  Thus, $C$ is bounded.  Hence it is compact.

Let $\mathcal{K}^*$ be the set of all non-empty compact subsets of $\R^2$, and let 
$d_H$ be the Hausdorff metric on $\mathcal{K}^*$.  That is, 
\[
d_H(A, B) = \max\{\max_{q \in a} d(q, B), \max_{q \in B} d(q, A)\}.
\]
It follows that a non-empty compact set $X \subseteq \R^2$ is computably compact if and only if it is possible to uniformly compute from a number $k \in \N$ a non-empty finite $E \subseteq \Q^2$ such that 
$d_H(E, X) < 2^{-k}$.  See, \emph{e.g.}, Section 5.2 of \cite{Weihrauch.2000}.  At the same time, by construction, 
\[
d_H(C_t, C_{t+1}) < 2^{-t}.
\]
So, when $t_1 > t$, 
\[
d_H(C_t, C_{t_1}) < 2^{-t + 1}.
\]
We now calculate an upper bound on $d_H(C, C_t)$.  By construction, if $q \in C$, then 
$d(q, C_t) < 2^{-t}$.  
So, $\max_{q \in C}, d(q, C_t) \leq 2^{-t}$.  Since $C_t \subseteq C$, $\max_{q \in C_t} d(q, C) = 0$. 
Hence, $d_H(C, C_t) \leq 2^{-t}$.

Each $C_t$ is computably compact uniformly in $t$.  That is, from $t,k \in \N$ it is possible to uniformly compute a finite $E_{t,k} \subseteq \Q^2$ such that $d_H(E_{t,k}, C_t) < 2^{-k}$.  So, 
\begin{eqnarray*}
d_H(E_{t+1, t+1}, C) & \leq & d_H(E_{t+1, t+1}, C_{t+1}) + d_H(C_{t+1}, C)\\
&< & 2^{-(t+1)} + 2^{-(t+1)} = 2^{-t}
\end{eqnarray*}
So, $C$ is computably compact.
\end{proof}

\begin{lemma}\label{lm:ELC}
$C$ is effectively locally connected.
\end{lemma}

\begin{proof}
Suppose we are given as input a name of an open $U \subseteq \R^2$ and a name of a point $q \in U \cap C$.  We show that from these data it is possible to uniformly compute a name of an open set $V \subseteq U$ such that $q \in V$ and $V \cap C$ is connected. 

We first introduce some notation.  Namely, if $R_e$ acts at stage $t+1$, then let 
\[
B_{e, t+1} = \bigcup_{e' >e} S_{e', t+1}.
\]
If $R_e$ does not act at $t+1$, then let $B_{e, t+1} = \emptyset$.  Also, let $B_{e, 0} = \emptyset$.  By construction, $B_{e,t+1} \cap C$ is connected.  

From the given data, it is possible to compute a positive rational number $\epsilon$ such that $D_\epsilon(q) \subseteq U$.

We now read the name of $q$ and cycle through all natural numbers until we find $e, t_0 \in \N$ and a rational rectangle $F_0$ such that $q \in F_0$, $2^{-t_0 + 1} + \diam(F_0) < \epsilon$, and one of the following holds.
\begin{enumerate}
	\item $F_0 \cap C_{t_0} \subseteq h_{e,t_0} - v_{t_0}$, $\overline{F_0} \subseteq S_{t_0,e}$, and 
	$d(\overline{F_0}, v_{t_0}) > 2^{-t_0}$.\label{itm:A}
	
	\item $F_0 \cap C_{t_0} \subseteq v_{t_0} - h_{e,t_0}$, $\overline{F_0} \subseteq S_{t_0,e}$, and 
	$d(\overline{F_0}, h_{e,t_0}) > 2^{-t_0}$.\label{itm:B}
	
	\item $q_{e,t_0} \in F_0$ and $\overline{F_0} \subseteq S_{t_0,e}$.\label{itm:C}
	
	\item $F_0 \cap S_{t_0,e} \cap S_{t_0, e+1} \cap C_{t_0} \neq \emptyset$, 
	$F_0 \cap C_{t_0} \subseteq v_{t_0}$, 
	$\overline{F_0} \subseteq S_{t_0,e} \cup S_{t_0,e+1}$, and 
	$d(\overline{F_0}, h_{e,t_0})$, $d(\overline{F_0}, h_{e+1, t_0}) > 2^{-t_0}$.\label{itm:D}
	
	\item $\overline{F_0} \subseteq S_{e,t_0}$ and $p_{t_1 + 1} \in F_0$ for some $t_1 \leq t_0$.\label{itm:E}
\end{enumerate}
Since $q \in C$, this process must terminate.

If one of (\ref{itm:A}) - (\ref{itm:D}) holds, let 
\[
V = \bigcup\{ B_{e_1, t_1 + 1}\ :\ e_1 \in \N\ \wedge\ t_1 \geq t_0\ \wedge\ B_{e_1, t_1 + 1} \cap F_0 \neq \emptyset\}\ \cup F_0.
\]
If (\ref{itm:E}), then let 
\[
V = \bigcup\{S_{t_0,e'}\ :\ e' > e\ \wedge \diam(S_{t_0,e'}) < 2^{-t_0}\}\ \cup\ F_0.
\]
Suppose one of (\ref{itm:A}) - (\ref{itm:D}) holds.  Then, $V$ is open and a name of $V$ can be uniformly computed from the given data.  It also follows that 
\[
V \cap (\bigcup_s C_s) = \bigcup\{ B_{e_1, t_1 + 1} \cap C_{t_1 +1}\ :\ e_1 \in \N\ \wedge\ t_1 \geq t_0\ \wedge\ B_{e_1, t_1 + 1} \cap F_0 \neq \emptyset\}\ \cup\ (F_0 \cap C_{t_0}).
\]
In each case, $F_0 \cap C_{t_0}$ is connected.  Suppose $e_1 \in \N$, $t_1 \geq t_0$, and 
$B_{e_1, t_1 + 1} \cap F_0 \neq \emptyset$.  By inspection of cases, $B_{e_1, t_1 + 1} \cap C_{t_0} \neq \emptyset$.  By Theorem 1-14 of \cite{Hocking.Young.1988}, $V \cap \bigcup_s C_s$ is connected.  It then follows from Proposition \ref{prop:LIMIT.POINTS} that $V \cap C$ is connected.  

Finally, it follows that the diameter of $V$ is at most $2^{-t_0 + 1} + \diam(F_0)$.  Hence, 
since $q \in V$, $V \subseteq U$.  

Suppose (\ref{itm:E}).  Again, $V$ is open and it is possible to uniformly compute a name of $V$ from the given data.
Also, $q \in V$.  The diameter of $V$ is at most $2^{-t_0}$.  Hence, $V \subseteq U$.  By construction, 
\[
C \cap \bigcup\{S_{t_0, e'}\ :\ e' > e\ \wedge\ \diam(S_{t_0, e'}) < 2^{-t_0}\}
\]
is connected.  Since (\ref{itm:E}), $R_e$ acts at $t_1 + 1$.  Hence, $F_0 \cap C_{t_1}$ is connected.  
By Proposition \ref{prop:LIMIT.POINTS}, 
\[
\{p_{t_1 +1 }\} \cup C \cap \bigcup\{S_{t_0, e'}\ :\ e' > e\ \wedge\ \diam(S_{t_0,e'}) < 2^{-t_0}\}
\]
is connected.  But, $p_{t_1 + 1} \in R_0 \cap C_{t_1}$. Hence, $V \cap C$ is connected.
\end{proof}

Thus, by Theorem \ref{thm:PEANO}, $C$ is a computable curve.  However, Lemma \ref{lm:ELC} is not merely a roundabout way of reaching this conclusion, but is in fact a stronger conclusion.  See  \cite{Couch.Daniel.McNicholl.2010}.

\begin{lemma}\label{lm:COMP.RECT}
$C$ is rectifiable and its length is computable.
\end{lemma}

\begin{proof}
By construction, for each $t$ for which $C_{t+1} - C_t \neq \emptyset$, $\overline{C_{t+1} - C_t}$ is a curve.  Suppose $R_e$ acts at $t + 1$.  It follows that the length of $h_{e', t+1}$ is smaller than $2^{-t}2^{-(e' - e) - 1}$ whenever $e' > e$ and that the length of $\overline{C_{t+1} - C_t}$ is at most $2^{-t+2}$.  It follows that the length of $C$ is finite and that for each $t$, the length of $C$ differs from the length of $C_t$ by at most $2^{-t + 3}$.  Thus, the length of $C$ is computable.
\end{proof}

The proof is complete.
\end{proof}

\section*{Acknowledgement}

I thank Jack Lutz and the Computer Science Department of Iowa State University for their hospitality.  I also thank Lamar University for generously supporting my visit to Iowa State University during the Fall of 2010.  I also thank Jack for stimulating conversation and for one of the proofs of Theorem \ref{thm:NOT.ON.RECT}.  Finally, I thank my wife Susan for support.

\bibliographystyle{amsplain}
\bibliography{/Users/tmcnicho/myfolders/research/bibliographies/computability,/Users/tmcnicho/myfolders/research/bibliographies/analysis}

\def\cprime{$'$}
\providecommand{\bysame}{\leavevmode\hbox to3em{\hrulefill}\thinspace}
\providecommand{\MR}{\relax\ifhmode\unskip\space\fi MR }
\providecommand{\MRhref}[2]{%
  \href{http://www.ams.org/mathscinet-getitem?mr=#1}{#2}
}
\providecommand{\href}[2]{#2}
\begin{thebibliography}{10}

\bibitem{Boolos.Burgess.2002}
George~S. Boolos, John~P. Burgess, and Richard~C. Jeffrey, \emph{Computability
  and logic}, fourth ed., Cambridge University Press, Cambridge, 2002.

\bibitem{Brattka.2008}
Vasco Brattka, \emph{Plottable real number functions and the computable graph
  theorem}, SIAM J. Comput. \textbf{38} (2008), no.~1, 303--328.

\bibitem{Cooper.2004}
S.~Barry Cooper, \emph{Computability theory}, Chapman \& Hall/CRC, Boca Raton,
  FL, 2004.

\bibitem{Couch.Daniel.McNicholl.2010}
P.J. Couch, B.D. Daniel, and T.H. McNicholl, \emph{Computing space-filling
  curves}, To appear in Theory of Computing Systems.

\bibitem{Daniel.McNicholl.2009}
D.~Daniel and T.H. McNicholl, \emph{Effective local connectivity properties},
  Submitted.

\bibitem{Downey.Hirschfeldt.2010}
Rodney~G. Downey and Denis~R. Hirschfeldt, \emph{Algorithmic randomness and
  complexity}, Theory and Applications of Computability, Springer, New York,
  2010.

\bibitem{Falconer.2003}
Kenneth Falconer, \emph{Fractal geometry}, second ed., John Wiley \& Sons Inc.,
  Hoboken, NJ, 2003, Mathematical foundations and applications.

\bibitem{Garey.Johnson.1979}
Michael~R. Garey and David~S. Johnson, \emph{Computers and intractability}, W.
  H. Freeman and Co., San Francisco, Calif., 1979, A guide to the theory of
  NP-completeness, A Series of Books in the Mathematical Sciences.

\bibitem{Gu.Lutz.Mayordomo.2006}
X.~Gu, J.~Lutz, and E.~Mayordomo, \emph{Points on computable curves},
  Proceedings of the Forty-Seventh Annual IEEE Symposium on Foundations of
  Computer Science (Berkeley, CA), IEEE Computer Society Press, October 2006,
  pp.~469--474.

\bibitem{Gu.Lutz.Mayordomo.2009}
X.~Gu, J.H. Lutz, and E.~Mayordomo, \emph{Curves that must be retraced}, To
  appear.

\bibitem{Hocking.Young.1988}
John~G. Hocking and Gail~S. Young, \emph{Topology}, second ed., Dover
  Publications Inc., New York, 1988.

\bibitem{Jones.1990}
Peter~W. Jones, \emph{Rectifiable sets and the traveling salesman problem},
  Invent. Math. \textbf{102} (1990), no.~1, 1--15.

\bibitem{Ko.1995}
Ker-I Ko, \emph{A polynomial-time computable curve whose interior has a
  nonrecursive measure}, Theoret. Comput. Sci. \textbf{145} (1995), no.~1-2,
  241--270.

\bibitem{Lutz.2003}
Jack~H. Lutz, \emph{The dimensions of individual strings and sequences},
  Inform. and Comput. \textbf{187} (2003), no.~1, 49--79.

\bibitem{Lutz.2005}
\bysame, \emph{Effective fractal dimensions}, MLQ Math. Log. Q. \textbf{51}
  (2005), no.~1, 62--72.

\bibitem{Miller.2004}
J.~Miller, \emph{Degrees of unsolvability of continuous functions}, The Journal
  of Symbolic Logic \textbf{69} (2004), 555--584.

\bibitem{Moore.Kline.1919}
R.L. Moore and J.R. Kline, \emph{On the most general plane closed point-set
  through which it is possible to pass a simple continuous arc}, Ann. of Math.
  \textbf{20} (1919), no.~3, 218--223.

\bibitem{Okikiolu.1992}
Kate Okikiolu, \emph{Characterization of subsets of rectifiable curves in
  {${\bf R}^n$}}, J. London Math. Soc. (2) \textbf{46} (1992), no.~2, 336--348.

\bibitem{Sagan.1994}
Hans Sagan, \emph{Space-filling curves}, Universitext, Springer-Verlag, New
  York, 1994.

\bibitem{Soare.1987}
R.I. Soare, \emph{Recursively enumerable sets and degrees}, Springer-Verlag,
  Berling, Heidelberg, 1987.

\bibitem{Weihrauch.2000}
Klaus Weihrauch, \emph{Computable analysis}, Texts in Theoretical Computer
  Science. An EATCS Series, Springer-Verlag, Berlin, 2000.

\end{thebibliography}

\end{document}